\documentclass[11pt,reqno]{amsart}

\usepackage{hyperref}
\usepackage{amssymb}
\usepackage{xcolor}
\usepackage{xypic}
\xyoption{all}
\def\umono{\ar@{_{(}->}[u]}
\def\uumono{\ar@{_{(}->}[uu]}

\def\lmono{\ar@{_{(}->}[l]}
\def\llmono{\ar@{_{(}->}[ll]}

\newtheorem{theorem}{Theorem}[section]
\newtheorem{proposition}[theorem]{Proposition}
\newtheorem{corollary}[theorem]{Corollary}

\newtheorem{definition}[theorem]{Definition}
\newtheorem{remark}[theorem]{Remark}
\newtheorem{example}[theorem]{Example}

\title[Conditional flatness, fiberwise localizations, and admissible reflections
]{Conditional flatness, fiberwise localizations, and admissible reflections
}

\author{Marino Gran}
\thanks{This work was supported by the Fonds de la Recherche Scientifique - FNRS under Grant n° CDR J.0080.23}
\address{Institut de Recherche en Math\'ematique et Physique\\
Chemin du Cyclotron 2\\
B1348 Louvain-la-Neuve, Belgique}
\email{marino.gran@uclouvain.be}

\author{J\'er\^{o}me Scherer}
\address{\'Ecole Polytechnique F\'ed\'erale de Lausanne, EPFL \\ Mathematics\\ Station 8, MA B3 455
CH-1015 Lausanne, Switzerland}
\email{jerome.scherer@epfl.ch}

\date{\today}
\begin{document}


\begin{abstract}
We extend the group-theoretic notion of conditional flatness for a localization functor to any  pointed category, and investigate 
it in the context of homological categories and of semi-abelian categories.
In the presence of functorial fiberwise localization analogous results to those obtained in the category of groups hold, and we
provide existence theorems for certain localization functors in specific semi-abelian categories. We 
prove that a Birkhoff subcategory of an ideal determined category yields a conditionally flat localization,
and explain how conditional flatness corresponds  to the property of admissibility of an adjunction from the point of view of categorical Galois theory.
Under the assumption of fiberwise localization we give a simple criterion to determine when a (normal epi)-reflection is a torsion-free reflection.
This is shown to apply in particular to nullification functors in any semi-abelian variety of universal algebras. 
We also relate semi-left-exactness for a localization functor $L$ with what is called right properness for
the $L$-local model structure.
\end{abstract}


\maketitle


\section*{Introduction}
\label{sec intro}

In \cite{CondiFlat} the notion of \emph{conditionally flat functor} was introduced by the second author and Farjoun in order to investigate pullback preservation 
properties related to homotopical localization functors. 
This was first done in the category of topological spaces, and then interpreted in the context of the category of groups, where short exact sequences replace fibration sequences. 
Given a fibration sequence
\begin{equation}\label{fibration} 
\xymatrix{F \ar[r] &E  \ar[r] & B}
\end{equation}  
of topological spaces and a morphism $X \rightarrow B$, a natural question is to study the properties of the original fibration \eqref{fibration} 
that are inherited by the pullback fibration along $X \rightarrow B$. 
Given a functor $L$, a fibration sequence \eqref{fibration} over a connected space $B$ is called \emph{$L$-flat} in analogy with the algebraic notion
if the sequence 
\begin{equation*}
\xymatrix{L(F) \ar[r] &L(E ) \ar[r] & L(B)}
\end{equation*}  
is again a homotopy fibration sequence. There are several examples of such preservation in the literature 
on localizations in homotopy theory (see \cite{MR51:1825, MR2581908}, for instance). 
A functor $L$ is then a \emph{conditionally flat functor} when any pullback of an $L$-flat fibration is again $L$-flat.
In the topological context the property of conditional flatness was shown to characterize nullification functors among localization functors, this is
\cite[Theorem~2.1]{CondiFlat}), which follows from Berrick and Farjoun's work in \cite{MR1997044} and relies on the existence
of fiberwise localization. The latter is a construction which had been around already before 1980, but which May brought forward in \cite{MR554323}, 
noticing the key role it played in Sullivan's paper on the Adams conjecture \cite{MR442930}.
 
 When moving from the category of topological spaces to the category $\mathsf{Grp}$ of groups, fibration sequences were replaced by short exact sequences 
 \begin{equation*}
\xymatrix{0 \ar[r] & K \ar[r] &A   \ar[r] & B \ar[r] & 0,
}\end{equation*}
and the ``flatness property'' of such exact sequences was then considered with respect to pullbacks along group homomorphisms $C \rightarrow B$.
In the case of groups, a major difference with the case of topological spaces is then that conditional flatness of a functor no longer characterizes nullification functors. 
These are conditionally flat, but so are all localization functors associated to a variety of groups in the sense of \cite{MR0215899}.

The aim of the present paper is then twofold. On the one hand we widely extend the context from the category of groups to the abstract one of semi-abelian category,
(in the sense of Janelidze, M\'arki, and Tholen \cite{JMT}), thus including 
many algebraic examples such as the categories of rings, Lie algebras, crossed modules, compact groups \cite{TopologicalSemiAb} and cocommutative 
Hopf algebras \cite{GSV}. We show in Proposition~\ref{prop:condiflatfiberwise} that conditional flatness is characterized, 
in a general context including any semi-abelian category, by the same
properties as in the category of groups, as long as a fiberwise version of the localization functor is available. 
On the other hand we establish a more useful criterion implying conditional flatness, namely
that pullbacks of $\eta_C\colon C \rightarrow L(C)$ along any regular epimorphism in $\mathcal C$ between $L$-local objects (i.e. 
objects lying in the reflective subcategory) should be inverted by~$L$. This property is well-known in category theory, and it has been 
used to investigate several adjunctions between algebraic categories. 
Indeed, this latter is exactly the property of \emph{admissibility} of a reflection from the point of view of categorical Galois theory \cite{MR1061480}, as  
in Janelidze and Kelly's \cite{JK} (see also \cite{EveraertGran}, and the references therein).
It is always true that a conditionally flat reflector induces an \emph{admissible} reflection in the sense of categorical Galois theory.


Moreover, this condition actually turns out to be \emph{equivalent} to the one of \emph{admissibility} under the assumption of 
functorial fiberwise localization (see Proposition \ref{prop:condiflatfiberwise2}).

We also offer one result which does not depend on the existence of fiberwise localization, in the case of a reflection onto a Birkhoff subcategory 
(i.e. when the reflective subcategory is closed under regular quotients and subobjects):

\medskip

\noindent
{\bf Theorem~\ref{thm:Birkhoff}.}
\emph{When $\mathcal X$ is a Birkhoff subcategory of an ideal determined category $\mathcal C$
 the corresponding pointed endofunctor $L = U F \colon \mathcal C \rightarrow \mathcal C$ is
conditionally flat.}

\medskip

This result then applies to many interesting examples, since any subvariety of a semi-abelian variety \cite{BournJanelidze} provides an example of Birkhoff subcategory. 

Under the assumption of fiberwise localization we then characterize torsion-free reflections $F \colon \mathcal C \rightarrow \mathcal X$ in any homological category $\mathcal C$ among (normal epi)-reflections in terms of the property of stability under extensions of $\mathcal X$ in $\mathcal C$ (Proposition \ref{lem:barPAPA}). This result applies in particular to any nullification functor in a semi-abelian variety of universal algebras \cite{BournJanelidze} (Corollary \ref{thm:nullificationproper}).
Note that, unlike in the abelian case, in the semi-abelian context a (normal epi)-reflective subcategory stable under extensions is not necessarily a torsion-free subcategory, as Janelidze and Tholen observed in \cite{MR2342832}.

In the final section of the paper we adopt a homotopical viewpoint and revisit the results in terms of model categorical
properties. Preservation of $L$-flatness under pullbacks is related here to semi-left-exactness, aka right properness. This is already
present in the pioneering work by Cassidy, H\'ebert, and Kelly, \cite{MR779198}, as explained by Rosick\'{y} and Tholen in \cite{MR2369170}.
Let us mention also the article \cite{MR3071143} by Wendt where right properness of the $L$-local model structure is explicitly related to
the work of Berrick and Farjoun, \cite{MR1997044}.
It is also interesting to remark that the $\infty$-analogues of semi-left-exact-localizations studied by Gepner and Kock in \cite{MR3641669}
correspond to so-called locally cartesian localizations. In this setting fiberwise methods (localization in slice categories) are always at hand and 
heavily used.

The article is written for a public of both category theorists and topologists.


\medskip

\noindent {\bf Acknowledgements.} This work started when the second
author visited the first author in the Universit\'e catholique de Louvain-la-Neuve in 2015. It benefitted from
conversations with Olivia Monjon and Florence Sterck when we finally got back to it in the Fall of 2021.
We also thank the anonymous referee for several useful remarks and suggestions that improved the presentation of the article, and in particular for Remark \ref{rem:CassidyHebertKelly}.
\section{Regular, homological and semi-abelian categories}
\label{sec:setup}

\subsection{Regular categories}
\label{subsec:regular}
Recall that a finitely complete category $\mathcal C$ is a \emph{regular category} if two properties are satisfied:
 \begin{itemize}
 \item any morphism can be factored into a regular epimorphism (i.e. a coequalizer
of a pair of parallel morphisms) followed by a monomorphism ;
\item regular epimorphisms in $\mathcal C$ are stable under pullbacks: this 
means that the arrow $\pi_1$ in a pullback
\[
\xymatrix{{E \times_B A} \ar[r]^-{\pi _2} \ar@{->>}[d]_{\pi_1} & A  \ar@{->>}[d]^{f} \\
E \ar[r]_p& B
}
\]
is a regular epimorphism whenever $f$ is a regular epimorphism. 
\end{itemize}
Recall that in a regular category, regular epimorphisms compose and, moreover, if a composite $g \circ f$
is a regular epimorphism, then so is~$g$.

\subsection{Homological categories}
\label{subsec:homological}
When a regular category $\mathcal C$ is pointed (i.e. it has a zero object, denoted by $0$) one says that it is \emph{homological} \cite{BB}  if the \emph{Split Short Five Lemma} holds in $\mathcal C$: 
given a commutative diagram 
\[
\xymatrix@=12pt{0\ar[rr]_{}&&K\ar[rr]^{\mathsf{ker}(f)}\ar[dd]_{u}&&A\ar[dd]^{v}
\ar[rr]^{f}&& B \ar@<1ex>[ll]^s \ar[dd]^{w}\\
&&&&&\\ 0\ar[rr]^{}&&K'\ar[rr]_{\mathsf{ker}(f')}
&&A'\ar[rr]^{f'}&&B' \ar@<1ex>[ll]^{s'} }
\]
in $\mathcal C$ where $f \circ s = 1_B$ and $f' \circ s' = 1_{B'}$, $\mathsf{ker}(f)$ is the kernel of $f$ and $\mathsf{ker}(f')$ is the kernel of $f'$, then $v$ is an isomorphism whenever $u$ and $w$ are isomorphisms.
It is well-known that this assumption implies in particular that any regular epimorphism is a cokernel (so that regular epimorphisms
coincide with normal epimorphisms), and then the classical Short Five Lemma 
holds in a homological category. This implies in particular the validity of the following useful
\begin{proposition}\label{pullback-right}\cite{BB}
Given a commutative diagram of short exact sequences in a homological category $\mathcal C$
\[
\xymatrix@=12pt{0\ar[rr]_{}&&K\ar[rr]^{\mathsf{ker}(f)}\ar[dd]_{u}&&A\ar[dd]^{v}
\ar[rr]^{f}&& B  \ar[dd]^{w} \ar[rr] & & 0\\
&&&&&\\ 0\ar[rr]^{}&&K'\ar[rr]_{\mathsf{ker}(f')}
&&A'\ar[rr]_{f'}&&B' \ar[rr] & & 0 }
\]
$u$ is an isomorphism if and only if the right-hand commutative square is a pullback.
\end{proposition}

\subsection{Semi-abelian categories}
\label{subsec:semi-abelian}
A \emph{semi-abelian} category \cite{JMT} is a finitely cocomplete
 homological category such that
every (internal) equivalence relation in $\mathcal C$ is a kernel pair. This means that $\mathcal C$ is an exact category (in the sense of Barr).
Among the many examples of semi-abelian categories there are the category of groups, 
Lie algebras, crossed modules \cite{JMT}, compact Hausdorff groups \cite{TopologicalSemiAb}, non-unital rings, loops \cite{TopologicalSemiAb}, cocommutative Hopf algebras over a field \cite{GSV}, 
non-unital $\mathsf{C}^*$-algebras \cite{Gran-Ros}, Heyting semilattices \cite{Johnstone}, etc.

\section{Localization and factorization systems}
\label{sec:factorization}
In this section we work with a semi-abelian category $\mathcal C$ as defined above, even though many of the facts we recall
now are valid in a more general setting. The main references here are Bousfield's \cite{MR57:17648} for the homotopy theory
viewpoint and the article \cite{MR779198} by Cassidy, H\'ebert, and Kelly for the categorical side.

\subsection{Factorization systems}
\label{subsec:factorization}
A \emph{prefactorization system} in $\mathcal C$ consists of classes of maps $\mathcal E$ and $\mathcal M$ determining each other
by \emph{unique} lifting properties or orthogonality properties. Thus a morphism $f$ belongs to $\mathcal E$ if and only if there is a unique
filler in any commutative square
\[
\xymatrix{
A \ar@{->}[r] \ar@{->}[d]_{f} & X \ar[d]^{p} \\
B\ar@{->}[r] \ar@{.>}[ur] & Y
}
\]
where $p$ belongs to $\mathcal M$. We write $\mathcal E = {}^{\perp} \mathcal M$. Dually, $\mathcal M = \mathcal E^{\perp}$.
A \emph{factorization system} is a prefactorization system where every map can be factored into a morphism in $\mathcal E$
followed by a morphism in $\mathcal M$.

\subsection{Localization and reflectors}
\label{subsec:localization}
A pointed endofunctor $(L \colon \mathcal C \rightarrow \mathcal C,\eta \colon 1_{\mathcal C} \Rightarrow L)$ is called idempotent 
if $L\eta \colon L \rightarrow LL$ is an isomorphism and $L\eta =\eta L$. It is common in algebraic topology to call \emph{localization} 
an idempotent pointed endofunctor $(L ,\eta )$, and we shall adopt this terminology. Note, however, that in category theory the meaning of the term ``localization'' is quite different, meaning a reflective subcategory where the reflector preserves (finite) limits. In this context our localization functors are usually called idempotent monads since the inverse of $\eta L = L \eta$ defines a monad multiplication.

A factorization system $({\mathcal E}, {\mathcal M})$ gives rise to a localization functor $L: \mathcal C \rightarrow \mathcal C$ by factoring the morphism
$X \rightarrow 0$ as explained in \cite[2.5]{MR57:17648}. This is a co-augmented and
idempotent functor, and any object $X$ comes with a natural morphism $\eta_X \colon X \rightarrow L(X)$ to an object $L(X)$ having
the property that $L(X ) \rightarrow 0$ belongs to $\mathcal M$. Such an object is called \emph{$L$-local}.

Conversely, when $\mathcal C$ is finitely well-complete, a localization functor $L$ yields a factorization system with $\mathcal E(L)$ consisting of 
all $L$-equivalences,  i.e. morphisms turned into isomorphisms by $L$, and $\mathcal M(L) = \mathcal E(L)^{\perp}$. This is due to
Cassidy, H\'ebert, and Kelly in \cite[Corollary~3.4]{MR779198} (see also the more recent article by Salch, \cite[Theorem~3.4]{MR3606285},
where the author already rephrased the original results).

\medskip

\begin{remark}
\label{rem:CassidyHebertKelly}
{\rm There is a well-known one-to-one correspondence between idempotent monads and full reflective subcategories. However there is no correspondence between localization functors and factorization systems, as shown in \cite{MR779198}. If one associates to a factorization system its canonical localization functor, and then apply the above construction to get a factorization system back, one does not in general recover the original factorization system, but its \emph{reflective interior}.}   
\end{remark}

\subsection{Birkhoff subcategories}
\label{subsec:Birkhoff}
In our work we shall also be interested in the situation where $\mathcal X$ is a \emph{Birkhoff subcategory} of a category $\mathcal C$
\begin{equation*}
\xymatrix@=20pt{ {\mathcal X \, } \ar@<-1ex>[r]_-{U}^{\perp} & {\, \mathcal C \, }
\ar@<-2ex>[l]_F}
\end{equation*}
where $U$ is the inclusion functor, and $F$ its left adjoint.
Being a Birkhoff subcategory means that $\mathcal X$ is a full (replete) and (regular epi)-reflective subcategory of $\mathcal C$ with the additional property 
that it is closed in $\mathcal C$ under regular quotients. Accordingly,
each component  $\eta_A \colon A \rightarrow UF(A)$ of the unit of the adjunction is a regular epimorphism and, moreover, $\mathcal X$ is also 
stable in $\mathcal C$ under regular quotients: if $\xymatrix{A \ar@{->>}[r]^f & B }$ is a regular epimorphism in $\mathcal C$ with $A$ in $\mathcal X$, 
then $B$ also belongs to $\mathcal X$.

\begin{example}
\label{ex:Birkhoff}
{\rm Any subvariety $\mathcal X$ of a variety $\mathcal C$ of universal algebras is a Birkhoff subcategory by the classical Birkhoff Theorem. 
This applies to many situations: by adding \emph{any} identity to the ones of a given algebraic theory one always determines a Birkhoff subcategory. 
This includes of course the classical examples of abelian groups or, more generally, of nilpotent or of solvable groups of a fixed class $\leq c$ in the category $\mathsf{Grp}$ of groups.
}
\end{example}

\section{Conditional flatness}
\label{sec:condi}
Our aim in this section is to study the notions of flatness and conditional flatness associated with a localization functor 
in a semi-abelian category. In analogy to the algebraic notion of flatness (tensoring by a flat ring preserves exactness), flatness
for homotopy functors was defined by Farjoun and the second author in \cite{CondiFlat} in terms of preservation of fibration sequences. 
The same was done in the category of groups in terms of preservation of extensions. 


\label{subsec:flat}
In a pointed category $\mathcal C$ one can translate this 
definition as follows.


\begin{definition}
\label{def:flat}
{\rm An extension $0 \rightarrow K \rightarrow E \rightarrow Q \rightarrow 0$
is  \emph{$L$-flat} if the functor $L \colon \mathcal C \rightarrow \mathcal C$ sends it again to an extension: $0 \rightarrow L(K) \rightarrow L(E) \rightarrow L(Q) \rightarrow 0$.
}
\end{definition}
The definition of conditional flatness from \cite{CondiFlat} still makes sense in any pointed category:

\begin{definition}
\label{def:condiflat}
{\rm A functor $L \colon \mathcal C \rightarrow \mathcal C$ in a pointed category $\mathcal C$ is \emph{conditionally flat} if any pullback of an $L$-flat extension is again 
$L$-flat.}
\end{definition}

\subsection{Fiberwise localization}
\label{subsec:fiberwise}

\begin{definition}
\label{def:fiberwise}
{\rm  Given a functor $L \colon \mathcal C \rightarrow \mathcal C$ in a pointed category $\mathcal C$ we say that an extension 
$0 \rightarrow K \rightarrow E \rightarrow Q \rightarrow 0$ in $\mathcal C$ admits a \emph{fiberwise localization} 
if there is a commutative diagram of horizontal extensions
\begin{equation}\label{construction}
\xymatrix{
0 \ar@{->}[r]  & K  \ar@{->}[r] \ar@{->}[d]^{\eta_K} & E  \ar@{->}[r] \ar@{->}[d]^e & Q  \ar@{->}[r] \ar@{=}[d] & 0   \\
0 \ar@{->}[r] & L(K) \ar@{->}[r]  & \overline E  \ar@{->}[r]  & Q  \ar@{->}[r]  & 0
}
\end{equation}
where $e : E \rightarrow \overline E$ is inverted by $L$. If the assignment $E \rightarrow \overline E$ is functorial (in the obvious sense) one says that it forms a \emph{functorial fiberwise localization} for $L$.
}
\end{definition}

Any localization in the category $\mathsf{Grp}$ of groups enjoys fiberwise localization as shown by 
Casacuberta and Descheemaeker \cite{MR2125447}. Other interesting examples will be considered at the end of this section,
but there are also localization functors in certain homological categories that do not admit fiberwise versions, see \cite{MSS}.

\subsection{Pullbacks along reflections}
\label{subsec:pbalong}
We are now going to show that in any homological category the existence of a functorial fiberwise localization has an interesting consequence.
The following result refines and generalizes \cite[Proposition~4.1]{CondiFlat} from the category of groups to any homological category. The second, more amenable, condition describes admissible reflections in the sense of Janelidze-Kelly \cite{JK}, as we discuss in the next section.

\begin{proposition}
\label{prop:condiflatfiberwise}
Let ${\mathcal X }$ be a full reflective subcategory of a homological category $\mathcal C$ 
\[
\xymatrix@=20pt{ {\mathcal X \, } \ar@<-1ex>[r]_-{U}^{\perp} & {\, \mathcal C \, }
\ar@<-2ex>[l]_F}
\]
that admits a
functorial fiberwise localization. Then the following conditions are equivalent:
\begin{enumerate}
\item the corresponding localization $(L=UF, \eta)$ is conditionally flat;
\item the pullback of $\eta_C\colon C \rightarrow L(C)$
along any regular epimorphism in $\mathcal C$ between $L$-local objects is inverted by $L$.
\end{enumerate}
\end{proposition}

\begin{proof}
Condition $(1)$ clearly implies $(2)$, and let us then prove that $(2)$ implies conditional flatness of $L$. 
Let $0 \rightarrow K \rightarrow E \rightarrow Q \rightarrow 0$ be an $L$-flat extension, and $f : X \rightarrow Q$ any arrow. 
We first observe that, by applying fiberwise localization, there is no restriction in assuming that $K$ is $L$-local. 
In order to see this, consider the right-hand pullback along $f$ and the kernel $\kappa$ of $p_2$:
\[
\xymatrix{
& & P \ar@{->>}[r]^{p_2} \ar@{->}[d]^{p_1}  & X \ar[d]^f \ar[r] & 0\\
0 \ar[r] & K \ar[r]_k \ar@{.>}[ur]^{\kappa} & E\ar@{->>}[r]_p & Q \ar[r] & 0
}
\]
By using the functorial fiberwise localization of $L$ one gets the following commutative diagram of short exact sequences 
(here we use the same notations as in Definition \ref{def:fiberwise}):
\[
\xymatrix{
& & \overline{P}  \ar@{->>}[r] \ar@{->}[d] & X \ar[d]^{f} \ar[r] & 0\\
0 \ar[r] & L(K) \ar[r]  \ar[ur] & \overline{E} \ar@{->>}[r] & Q \ar[r] & 0
}
\]
In any homological category the right-hand square in the above diagram is then again a pullback (this follows from Proposition \ref{pullback-right}, since $\mathcal C$ is assumed to be a homological category). If we write $\pi \colon P \rightarrow \overline{P}$ for the $L$-equivalence 
in the construction \eqref{construction} of the exact sequence $0 \rightarrow L(K) \longrightarrow \overline{P} \longrightarrow X \rightarrow 0$
by fiberwise localization, one obviously has that $\eta_P \cong \eta_{\overline{P}} \circ \pi$, and this implies that the latter exact sequence is $L$-flat 
if and only if so is the sequence $0 \rightarrow K \longrightarrow P \longrightarrow X \rightarrow 0$.

The next step is to reduce the proof to the case of an extension of $L$-local objects. This is done by noticing that, since $K$ in the $L$-flat exact sequence 
$0 \rightarrow K \longrightarrow E \longrightarrow Q \rightarrow 0$ can already be assumed to be $L$-local, the square
\[
\xymatrix{
E \ar@{->>}[r] \ar@{->}[d] & Q \ar[d] \\
L(E) \ar@{->>}[r] & L(Q)
}
\]
is a pullback (we again use Proposition \ref{pullback-right}). Finally, we use the universal property of the localization and factor any map $X \rightarrow L(Q)$
through $\eta_X \colon X \rightarrow L(X)$ to decompose the pullback $P$ of $L(E) \rightarrow L(Q)$ and $X \rightarrow L(Q)$ as the composite of two pullbacks:
\[
\xymatrix{P \ar[r] \ar[d] &X \ar[d]  \\
P' \ar@{->>}[r] \ar[d] & L(X) \ar[d] \\
L(E) \ar@{->>}[r] & L(Q).
}
\]
Here $P'$ is a limit of $L$-local objects, hence $L$-local, and therefore yields, by regularity of $\mathcal C$, another regular epimorphism $P' \rightarrow L(X)$
of $L$-local objects. The upper square is of the form required in order to apply assumption $(2)$.
\end{proof}

\subsection{Existence of functorial fiberwise localization}
\label{subsec:examplesfiberwise}
Besides the example of the category $\mathsf{Grp}$ of groups, there are many other examples of categories admitting functorial fibrewise localizations in certain circumstances. We focus from here on localization functors for which the coaugmentation morphisms $\eta_X \colon X \rightarrow L(X)$ are always normal epimorphisms. In that case we write $t_X \colon T(X) \rightarrow X$ for the kernel of $\eta_X$ and identify the latter with the quotient map $X \rightarrow X/T(X)$.

\begin{proposition}\label{iff-c}
Let $\mathcal C$ be a homological category and $L \colon \mathcal C \rightarrow \mathcal C$ be a localization functor such that any coaugmentation morphism $\eta_X \colon X \rightarrow L(X)$ is a normal epimorphism. Then $\mathcal C$ admits functorial fiberwise localization (with respect to $L$) if and only if one of the following conditions holds:
\begin{enumerate}
    \item for any normal monomorphism $k \colon K \rightarrow E$ in $\mathcal C$, the pushout of 
$k$ along $\eta_K$ exists 
\begin{equation}\label{pushout-mono}
\xymatrix{
K \ar[r]^k \ar[d]_{\eta_K}& E \ar[d] \\
L(K) \ar[r]_{\overline{k}} & \overline{E}
}
\end{equation}
and this square is a pullback;
\item
for any normal monomorphism $k \colon K \rightarrow E$ in $\mathcal C$, the pushout \eqref{pushout-mono} of 
$k$ along $\eta_K$ exists and the morphism $\overline{k}$ is a monomorphism.
\end{enumerate}
\end{proposition}

\begin{proof}
Let us assume first that fiberwise localization exists.
Given any extension $$\xymatrix{0 \ar[r] & K \ar[r]^k & E  \ar[r]^f & B   \ar[r] & 0 }$$ a commutative diagram
$$\xymatrix@=35pt{
0 \ar[r] & K \ar[r]^k \ar[d]_{\eta_K}& E \ar[d]^{e} \ar[r]^f & B \ar@{=}[d]  \ar[r] & 0 \\
0 \ar[r] & L(K) \ar[r]_{\overline{k}} & \bar E \ar[r]^{g} & B \ar[r] &  0 
}
$$
exists by the assumption of  fiberwise localization. The left-hand square is clearly a pullback, and the middle vertical morphism $e$ is a regular epimorphism since the base category $\mathcal C$ is homological (see Proposition $8$ in \cite{Bourn-protomodularity}). In this context any pullback of a regular epimorphism along any morphism is a pushout (see \cite{Bourn-Original}), and this proves that $(1)$ holds.
It is clear that $(1)$ implies $(2)$ since pullbacks reflect monomorphisms in $\mathcal C$ \cite{Bourn-Original}.

Assume then that $(2)$ holds. In particular one has the lower left-hand pushout in the commutative diagram
\begin{equation}\label{third}
\xymatrix@=35pt{
0 \ar[r] & K \ar[r]^k \ar[d]_{\eta_K}& E \ar[d]^{\pi} \ar[r]^f & B \ar@{=}[d]  \ar[r] & 0 \\
0 \ar[r] & K/T(K ) \ar[r]_{\overline{k}} & \overline{E} \ar@{.>}[r]_{\overline{f}} & B \ar[r] &  0. \\
}
\end{equation}
where $\overline{k}$ is a monomorphism by assumption.
Since $f \circ k \circ t_K = 0$ the universal properties of the cokernel $\eta_K$, and then of the left-hand pushout, yield a unique arrow $\overline{f} \colon \overline{E} \rightarrow B$ making the right-hand square above commute. The canonical factorization $\phi$ from $K/T(K)$ to the kernel $\mathsf{Ker}(\overline{f})$ of $\overline{f}$ such that $\mathsf{ker}(\overline{f}) \circ \phi = \overline{k}$ is then a monomorphism (since $\overline{k}$ is a monomorphism). It is also a regular epimorphism, since so is $\phi \circ \eta_K$, this latter being the pullback of $\pi$ along $\mathsf{ker}(\overline{f})$. It follows that $\phi$ is an isomorphism, and the lower sequence in diagram \ref{third} is then exact.

The proof will be then complete if we show that $\pi$ is inverted by $L$.
For this, consider the commutative diagram
\begin{equation*}
\xymatrix@=35pt{
   & T(K)  \ar[d]_{T(k)} &   & & \\
0 \ar[r] & T(E)  \ar[r]^{t_E} \ar[d]_{q}& E \ar[d]^{\pi} \ar[r]^{\eta_E} & L(E)  \ar@{=}[d]  \ar[r] & 0 \\
0 \ar[r] & T(E)/T(K ) \ar@{.>}[r]_-j & \overline{E}  \ar@{.>}[r]_p & L(E) \ar[r] &  0 \\
   }
\end{equation*}
where 
\begin{itemize}
\item $T(k)$ is a normal monomorphism since so is $t_E \circ T(k) \colon T(K) \rightarrow E$ (the assumption that $\overline{k}$ is a monomorphism implies that $T(K)$ is the intersection $ T(E) \cap \mathsf{Ker}(\pi)$ of two normal monomorphisms) and $t_E$ is a monomorphism;
\item $q$ is the quotient of $T(E)$ by $T(K)$;
\item $j$ is the unique morphism such that $j \circ q = \pi \circ t_E$;
\item $p$ is the unique morphism such that $p \circ \pi = \eta_E$.
\end{itemize}


Now, if $f \colon \overline{E} \rightarrow A$ is any morphism with $A$ a local object, then there is a unique morphism $\psi \colon L(E) \rightarrow A$ such that $\psi \circ \eta_E = f \circ \pi$. This morphism $\psi$ is also the unique one such that $\psi \circ p = f$ (since $\pi$ is an epimorphism), proving that $p = \eta_{\overline{E}}$ (and $L(E) = L(\overline{E}$)), so that $\pi$ is indeed inverted by $L$. 
One can easily check that this construction is functorial, and this completes the proof.
\end{proof}

\begin{remark}\label{N}
{\rm In any homological category a property equivalent to properties $(1)$ and $(2)$ used in Proposition \ref{iff-c} consists in requiring that 
the monomorphism  $$k \circ t_K \colon T(K) \rightarrow K \rightarrow E$$ is normal. This latter was called condition $(N)$ in \cite{EveraertGran2}.
The fact that $(N)$ is equivalent to condition $(1)$ easily follows from Proposition \ref{pullback-right}, by choosing the quotient $E/T(K)$ as $\overline{E}$ in diagram \eqref{pushout-mono}. We had used condition $(N)$ in a previous version of the article, and we thank the referee for suggesting to also consider condition $(2)$ in the above proposition. The equivalent property $(N)$ will now be useful in the following examples.
}
\end{remark}


\begin{example}
{\rm Proposition \ref{iff-c} can be applied to any homological category, hence in particular to the category $\mathsf{Grp(Top)}$ of topological groups 
\cite{TopologicalSemiAb}, that has the property that any regular epimorphism is normal. Consider a localization functor 
$L \colon \mathsf{Grp(Top)} \rightarrow \mathsf{Grp(Top)}$ for which each coaugmentation morphism $\eta_X$ of the localization 
is a normal epimorphism (= open surjective group homomorphism). Given any short exact sequence 
\begin{equation}
\label{exte} 
\xymatrix{0 \ar[r] &  K \ar[r] & E \ar[r] & B \ar[r] &  0}
\end{equation}
in $\mathsf{Grp(Top)}$, by taking the kernel $t_K \colon T(K) \rightarrow K$ of the unit $\eta_K$ of any localization one obtains a 
\emph{characteristic} subgroup $T(K)$ of $K$ 
(see Example $2.2$ in \cite{EveraertGran2}, for instance). Accordingly, the subgroup $T(K)$ is also normal in $E$, hence condition $(N)$ in Remark \ref{N} holds, as desired.
The same observation also applies to the category $\mathsf{Grp(Haus)}$ of Hausdorff groups.
}
\end{example}

\begin{example}
{\rm Let then $\mathsf{Hopf}_{A, coc}$ be the category of cocommutative Hopf algebras over a field $A$, that was shown to be 
semi-abelian in \cite{GSV}. Given an extension \eqref{exte}, by using the same notations as above, 
the Hopf subalgebra $T(K) \rightarrow K \rightarrow E$ induced by any localization functor 
$L \colon \mathsf{Hopf}_{A, coc} \rightarrow \mathsf{Hopf}_{A, coc}$ is also a \emph{normal} Hopf subalgebra of $E$. 
Indeed, denote by $S \colon K \rightarrow K$ the antipode of $E$, and by $\phi_e \colon K \rightarrow K$ the map defined, 
for any $e \in E$, by $\phi_e (t) =  e_1 t S(e_2)$ for any $t \in K$. Here we use the usual Sweedler convention for Hopf algebras, 
so that we write $\Delta (e) = e_1 \otimes e_2$, with $\Delta \colon E \rightarrow  E \otimes E$ the comultiplication. 
This map $\phi_e \colon K \rightarrow K$ is easily seen to be a Hopf algebra morphism. By functoriality of the natural transformation 
$t \colon T \Rightarrow \mathsf{id}_{\mathcal C}$ it follows that $\phi_e$ restricts to $T(K)$, yielding a morphism $T(K) \rightarrow T(K)$. 
This means that, for any $t \in T(K)$,
$\phi_e (t) \in T(K)$, hence $T(K)$ is normal in $E$, and condition $(N)$ then holds. We conclude that, whenever we have a localization functor 
$L \colon \mathsf{Hopf}_{A, coc} \rightarrow \mathsf{Hopf}_{A, coc}$ with the property that the coaugmentation
morphism $\eta_X\colon X \rightarrow L(X)$ is a normal epimorphism, the category $\mathsf{Hopf}_{A, coc}$ admits fiberwise localization. 
For instance, the ``abelianization functor'' $\mathsf{ab} \colon \mathsf{Hopf}_{A, coc} \rightarrow 
\mathsf{Hopf}_{A, coc}^{comm}$ as described in Section $4$ in \cite{GSV}, necessarily yields a functorial fiberwise localization, with 
$L = U \, \mathsf{ab}  \colon  \mathsf{Hopf}_{A, coc} \rightarrow \mathsf{Hopf}_{A, coc} $ (here 
$U \colon \mathsf{Hopf}_{A, coc}^{comm} \rightarrow \mathsf{Hopf}_{A, coc}$ is the forgetful functor).
}
\end{example}

\begin{remark}
{\rm One might hope that similar results hold whenever one is dealing with a category of internal groups in finitely complete category, 
since the examples mentioned here above, such as groups, topological groups, Hausdorff groups, and cocommutative Hopf algebras, 
are of this type (the category $\mathsf{Hopf}_{A, coc}$ can also be seen as the category of internal groups in the category of cocommutative 
coalgebras). This is not the case, as it follows from the results in \cite{MSS}, where some counter-examples are given in the semi-abelian 
category $\mathsf{XMod}$ of crossed modules, that can be also seen (up to a category equivalence) as the category $\mathsf{Grp(Cat})$ of internal 
groups in the category $\mathsf{Cat}$ of (small) categories (see \cite{JMT}, for instance, and the references therein).}
\end{remark}

\section{Admissible reflectors with respect to Galois theory}
\label{sec:admissibleGalois}
The type of pullbacks which appear in Proposition~\ref{prop:condiflatfiberwise} are the ones
 appearing in categorical Galois theory, in the form presented in the article \cite{JK} by Janelidze and Kelly. In the whole section we work in a homological
category $\mathcal C$, where we fix a 
 full reflective subcategory $\mathcal X$ as in 
 \begin{equation*}
\xymatrix@=20pt{ {\mathcal X \, } \ar@<-1ex>[r]_-{U}^{\perp} & {\, \mathcal C \, }
\ar@<-2ex>[l]_F}
\end{equation*}
and the corresponding localization functor $L = UF\colon \mathcal C \rightarrow \mathcal C$.

\begin{definition}\cite{JK}
\label{def:admissible}
{\rm The reflector $F \colon {\mathcal C } \rightarrow \mathcal X$ is \emph{admissible} for the class of regular epimorphisms 
if it preserves any pullback of the form
\begin{equation}\label{admissible-pb}
\xymatrix{
P \ar@{->}[r] \ar@{->}[d] & U(E) \ar@{->>}[d]^{U(x)} \\
X\ar@{->}[r]_-{\eta_X} & UF(X)
}
\end{equation}
where $x \colon E \rightarrow FX$ is a regular epimorphism in $\mathcal X$. 
}
\end{definition}

One could also require $F$ to preserve pullbacks as above for \emph{any} morphism $x$ in $\mathcal X$, in which case we are looking at semi-left exact reflections as introduced by Cassidy, H\'ebert, and Kelly in \cite{MR779198}. We will come back to this in Section~\ref{sec:torsionfree}.

\begin{definition}\cite{MR779198}
\label{def:semileftexact}
{\rm 
Let $\mathcal X$ a (normal epi)-reflective subcategory of $\mathcal C$. Then $F$ is \emph{semi-left-exact}, i.e., $F$ preserves all pullbacks of the form
\[
\xymatrix{
 P \ar@{->>}[r]^-{p_2} \ar@{->}[d]_{p_1}  & U(E) \ar[d]^x\\
X \ar@{->>}[r]_{\eta_C} & UF(X)
}
\]
where $x$ is any morphism in $\mathcal X$.
}
\end{definition}

In other words, the morphism $P \rightarrow U(E)$ coincides with the $P$-component of the unit $\eta_P \colon P \rightarrow UF(P)$ of the adjunction (up to unique isomorphism). There are several equivalent ways to characterize admissibility as we recall in the next proposition. They hold in particular for all the examples of semi-localizations of semi-abelian categories given in \cite{SemiLoc}.

\begin{proposition}
\label{prop:admissibleequivalences}
Let $\mathcal C$ be a homological category, $F \colon {\mathcal C } \rightarrow \mathcal X$ a reflector and $(L=UF, \eta)$ the corresponding localization. The following conditions are then equivalent:
\begin{enumerate}
\item The reflector $F$ is admissible for the class of regular epimorphisms;
\item the pullback of $\eta_C\colon C \rightarrow L(C)$
along any regular epimorphism in $\mathcal C$ between $L$-local objects is inverted by $L$; 
\item the functor $L=UF \colon \mathcal C \rightarrow \mathcal C$ preserves any pullback of the form
\[
\xymatrix{
 C \times_{L(C)} X \ar@{->>}[r]^-{p_2} \ar@{->}[d]_{p_1}  & X \ar@{->>}[d]^g\\
C \ar@{->>}[r]_{\eta_C} & UF(C)
}
\]
where $g$ is a regular epimorphism in $\mathcal C$ between objects in $\mathcal X$.
\end{enumerate}
\end{proposition}

\begin{proof}
The equivalence $(2) \Leftrightarrow (3)$ is obvious, while the equivalence between $(3)$ and $(1)$ follows easily from the fact that the functor $U \colon \mathcal X \rightarrow \mathcal C$ reflects limits, since it is a fully faithful right adjoint.    
\end{proof}

 
\begin{proposition}
\label{prop:condiflatfiberwise3}
 If $L= UF \colon \mathcal C \rightarrow \mathcal C$ is conditionally flat then the reflector $F \colon \mathcal C \rightarrow \mathcal X$ is admissible for the class of regular epimorphisms.
\end{proposition}

\begin{proof}
Let $K$ be the object part of the kernel of the vertical morphism $U(x)$  in \eqref{admissible-pb}. This is the object part of a limit of a diagram 
lying in $\mathcal X$, hence it lies itself in $\mathcal X$. The extension
\[
\xymatrix{0 \ar[r] & K \ar[r] &  U(E) \ar[r]^-{Ux}  & UF(X)\ar[r]   &0 
}
\]
is thus $L$-flat. If $L$ is conditionally flat, the (induced) pullback extension
\[
\xymatrix{0 \ar[r] & K \ar[r] &  P \ar[r] &  X \ar[r] & 0}
\] 
must be $L$-flat as well. This means that $L=UF$ takes it to an extension
\[
\xymatrix{0 \ar[r] & K \ar[r] & UF(P) \ar[r] & UF(X )\ar[r] & 0}
\]
where $K \cong UF(K)$ remains unchanged since it lies in $\mathcal X$. This extension comes with a natural transformation to the original extension:
\[
\xymatrix{0 \ar[r] & K \ar[r] \ar@{=}[d]& UF(P) \ar[r] \ar@{.>}[d]^{} &  UF(X) \ar@{=}[d] \ar[r] & 0 \\
0 \ar[r] & K \ar[r] &  U(E) \ar[r] &  UF(X) \ar[r] & 0. }
\]
We conclude by the Short Five Lemma (see \cite{BB}) that the middle dotted arrow is an isomorphism, and the arrow $P \rightarrow U(E)$ in the pullback \eqref{admissible-pb} is then isomorphic to the unit $\eta_P \colon P \rightarrow UF(P)$. This means that the reflector is admissible with respect to the class of regular epimorphisms, as desired.
\end{proof}

We can also reinterpret Proposition~\ref{prop:condiflatfiberwise} as follows:

\begin{proposition}
\label{prop:condiflatfiberwise2}
Let $\mathcal C$ be a homological category and assume that the localization functor $L$ admits a
functorial fiberwise localization. Then the functor $L$ is conditionally flat if and only if it is admissible with respect to regular epimorphisms. \hfill{\qed}
\end{proposition}

\begin{example}
{\rm It is well-known that any Birkhoff subcategory of a semi-abelian category induces an admissible reflector with respect 
to regular epimorphisms \cite{JK, MR2388093}. Together with the remarks in Section $3.3$ this implies in particular that 
any Birkhoff subcategory of the category $\mathsf{Hopf}_{A, coc}$ of cocommutative Hopf algebras over a field $A$ induces 
a conditionally flat functor $L$. In the semi-abelian category $\mathsf{Grp(Comp)}$ its Birkhoff subcategory $\mathsf{Grp(Prof)}$ 
of profinite groups also induces a conditionally flat functor, since the adjunction is admissible \cite{EveraertGran}. }
\end{example}

\section{The case of Birkhoff subcategories}
\label{sec:epireflections}
Let us restrict our attention to a Birkhoff subcategory $\mathcal X$ of a regular category $\mathcal C$
as in Subsection~\ref{subsec:Birkhoff}.
%
%
The suitable context to obtain the result of this section is the one of \emph{ideal determined} categories, as introduced in \cite{IDC}
by Janelidze, Mark\'{\i}, Tholen, and Ursini. These are regular categories $\mathcal C$ with binary coproducts such that 
\begin{enumerate}
\item any regular epimorphism in $\mathcal C$ is \emph{normal} (i.e. a \emph{cokernel});
\item \emph{normal monomorphisms are stable under images}: in any commutative square 
\[
\xymatrix{{A\, }\ar@{->>}[d]_f  \ar@{>->}[r]^a & {A'}\ar@{->>}[d]^{f'} \\
{B\, } \ar@{>->}[r]_{b} & B'
}
\]
in $\mathcal C$ where $f$ and $f'$ are normal epimorphisms, $a$ is a normal monomorphism and $b$ is a monomorphism, 
then $b$ is also a normal monomorphism. 
\end{enumerate}
As explained in \cite{IDC} any semi-abelian category is ideal determined. In particular all the examples mentioned before 
(groups, loops, rings, commutative algebras, associative algebras, cocommutative Hopf algebras, crossed modules, compact groups, ${\mathsf C}^{\star}$-algebras, etc.) 
are ideal determined. There are also some examples of ideal determined varieties that are not semi-abelian, as for instance the variety 
of implication algebras \cite{GU}.

The following theorem gives a natural condition guaranteeing the conditional flatness of  the pointed endofunctor $L$, without the toolkit of fiberwise localization.

\begin{theorem}
\label{thm:Birkhoff}
When $\mathcal X$ is a Birkhoff subcategory of an ideal determined category $\mathcal C$
 the corresponding pointed endofunctor $L = U F \colon \mathcal C \rightarrow \mathcal C$ is
conditionally flat.
\end{theorem}

\begin{proof}
Let us prove that $L \colon \mathcal C \rightarrow \mathcal C$ is conditionally flat. We consider then an $L$-flat extension 
\begin{equation}\label{ext} 
\xymatrix{ 
0 \ar[r] & K \ar[r]^{k} & E \ar[r]^{f} & X\ar[r] & 0 }
\end{equation}
and a morphism $g\colon A \rightarrow X$ in $\mathcal C$. We construct the pullback of the original extension along $g$
and have to prove that this extension 
$0 \rightarrow K \longrightarrow P \xrightarrow {p_2} A \rightarrow 0$  is again $L$-flat.

We know that the induced arrow $L(p_2) \colon L(P) \rightarrow L(A)$ is a normal epimorphism, 
since the arrow $\eta_A \circ p_2 \colon P \rightarrow L(A)$ is a normal epimorphism, being a composite of two normal epimorphisms (see Subsection~\ref{subsec:regular}).
We are now going to prove that the arrow $L(K) \rightarrow L(P)$ is the kernel of $L(p_2) \colon L(P) \rightarrow L(A)$.
First observe that the arrow $L(K) \rightarrow L(P)$ is a monomorphism, since the arrow 
\[
L(K) \rightarrow L(P) \rightarrow L(E) = \xymatrix{L(K) \ar[r]^{L(k)} & L(E) }
\] 
is a  monomorphism (the original extension \eqref{ext} being $L$-flat). 
Since the category $\mathcal C$ is ideal determined and the square 
\[
\xymatrix{K\ar[d]  \ar@{->>}[r]^{\eta_K} & L(K)\ar[d] \\
P \ar@{->>}[r]_{\eta_P} & L(P)
}
\]
is commutative with $K \rightarrow P$ a normal monomorphism and $L(K) \rightarrow L(P)$ a monomorphism, it follows that the arrow $L(K) \rightarrow L(P)$
is a normal monomorphism as well. Consequently $L(K) \rightarrow L(P)$ is the kernel of its cokernel  $q\colon L(P) \rightarrow  Q$ in $\mathcal C$. 
However, this latter is isomorphic to $L(p_2) \colon L(P) \rightarrow L(A)$. Indeed, this follows from the fact that the functor $F \colon \mathcal C \rightarrow \mathcal X$ preserves cokernels (being a left adjoint) while $U \colon \mathcal X \rightarrow \mathcal C$ preserves them since $\mathcal X$ is closed in $\mathcal C$ under (regular) quotients by the Birkhoff assumption.
\end{proof}

\section{Fiberwise localizations and stability under extensions}
\label{sec:torsionfree}

In this section we show that, when $\mathcal C$ is homological, torsion-free reflections $F \colon \mathcal C \rightarrow \mathcal X$ can 
be characterized among (normal epi)reflections admitting fiberwise localization in terms of the property of \emph{stability under extensions} of $\mathcal X$ in $\mathcal C$. 
We recall that a torsion-free reflection is associated to a torsion theory, see for example \cite[Definition~1.1]{SemiLoc}. In particular, the only morphism
from a torsion object to a local object is the zero morphism.

Recall that a full (replete) subcategory $\mathcal X$ of a pointed category $\mathcal C$ is stable under extensions (in $\mathcal C$) if, given any short exact sequence 
\begin{equation}\label{ses}
\xymatrix{0 \ar[r] & K \ar[r] & X \ar[r] & Y \ar[r] & 0
}
\end{equation}
in $\mathcal C$ with $K$ and $Y$ in $\mathcal X$, then $X$ is also in $\mathcal X$.

\begin{proposition}
\label{lem:barPAPA}
Let $\mathcal C$ be a homological category, $\mathcal X$ a (normal epi)-reflective subcategory of $\mathcal C$ with the property that the reflector $F \colon \mathcal C \rightarrow \mathcal X$ admits fiberwise localization. Let us write $T(X)$ for the kernel of the $X$-reflection $\eta_X \colon X \rightarrow F(X)$ of any $X$ in $\mathcal C$. Then the following conditions are equivalent:
\begin{enumerate}
\item $\mathcal X$ is stable in $\mathcal C$ under extensions;
\item $F (T (X)) = 0$ for any object $X$ in~$\mathcal C$;
\item $F$ is semi-left-exact; 
\item $\mathcal X$ is a torsion-free subcategory in $\mathcal C$.
\end{enumerate}  
\end{proposition}

\begin{proof}
$(1) \Rightarrow (2)$ Let us consider the short exact sequence \begin{equation}\label{ses-can}
\xymatrix{0 \ar[r] & T(X) \ar[r]^-{t_X} & X \ar[r]^-{\eta_X} & F(X) \ar[r] &  0} 
\end{equation}
and the associated exact sequence $0 \rightarrow F (T(X)) \rightarrow \overline X \rightarrow F(X) \rightarrow 0$ that exists by the assumption of fiberwise localization.
Since  the subcategory $\mathcal X$ is closed in $\mathcal C$ under extensions, $\overline X$ is in $\mathcal X$, and the fiberwise morphism $X \rightarrow \overline X$
is therefore an $F$-equivalence to an object in $\mathcal X$. It must thus be $\eta_X\colon X \rightarrow F(X)$ (up to isomorphism). This implies
that the kernel $F (T(X))$ of the morphism $\overline{X} \rightarrow F(X)$ is zero.

\noindent $(2) \Rightarrow (3)$ and $(3) \Rightarrow (4)$ both follow from Theorem $4.12$ in \cite{MR2262518}.

\noindent $(4) \Rightarrow (1)$ We briefly recall the known argument showing that a torsion-free subcategory $\mathcal X$ is closed under extensions in $\mathcal C$. Given  a short exact sequence \eqref{ses} with $K$ and $Y$ in $\mathcal X$, consider the canonical short exact sequence \eqref{ses-can}, where $T(X)$ is torsion and $F(X)$ is torsion-free. Clearly, $T(X) \rightarrow X \rightarrow Y$ is the zero morphism, hence $t_X$ factors through $K$. Since $T(X)$ is a subobject of $K$, $T(X) \in \mathcal X$ ($\mathcal X$ is closed under subobjects). Since it is also in the torsion subcategory, $T(X) \cong 0$ and $X \cong F(X) \in \mathcal X$, as desired.
\end{proof}

\noindent Unlike in the abelian case, in homological categories the property of stability under extensions of a (normal epi)-reflective subcategory $\mathcal X$ is not strong enough to guarantee that $F \colon \mathcal C \rightarrow \mathcal X$ is a reflector to a torsion-free subcategory, as observed in \cite{MR2342832}. The lemma above shows that, under the assumption of fiberwise localization, this is indeed the case.

\begin{remark}
\emph{From Proposition~\ref{prop:condiflatfiberwise} and Proposition \ref{lem:barPAPA} above we deduce that, under the assumption of functorial fiberwise localization, any semi-left-exact reflector $F \colon \mathcal C \rightarrow \mathcal X$ gives rise to a corresponding conditionally flat localization $L = UF \colon \mathcal C \rightarrow \mathcal C$. The converse does not hold however, even in the case of a (normal epi)-reflection associated to a Birkhoff subcategory, as
illustrated in the following classical example in the category of groups. }
\end{remark} 


\begin{example}
\label{ex:notproper}
{\rm Let us write $L_{ab}$ for the abelianization functor. 
The dihedral group $D_8$ of order $8$ abelianizes to $\mathbb Z/2\mathbb Z  \times \mathbb Z/2\mathbb Z$, an elementary abelian
$2$-group of rank two. Consider the following pullback in the category of groups:
\[
\xymatrix{
\mathbb Z/2\mathbb Z \ar@{->}[r] \ar@{->}[d] & 0 \ar[d] \\
D_8 \ar@{->}[r] & \mathbb Z/2 \mathbb Z\times \mathbb Z/2\mathbb Z
}
\]
The right hand side vertical morphism is a homomorphism of abelian groups
and the bottom morphism is the abelianization morphism of $D_8$. Its pullback however
is the map $\mathbb Z/2 \mathbb Z \rightarrow 0$, which is not the abelianization morphism for $\mathbb Z/2  \mathbb Z$.
Since fiberwise localization always exists in the category $\mathsf{Grp}$ of groups, Proposition~\ref{lem:barPAPA} applies and
tells us that the above problem reflects the fact that the subcategory $\mathsf{Ab}$ of abelian groups is not closed under extensions 
in $\mathsf{Grp}$, it is not torsion-free (in the categorical sense).
}
\end{example}

The fact that the abelianization functor is not semi-left-exact is well known. The ``relative version'' of the Galois theory 
developed by Janelidze \cite{MR1061480}, later also in collaboration with Kelly \cite{JK}, where the class of morphisms to be classified 
by the Galois theorem is the one of regular epimorphisms, was partly motivated by the possibility of applying their approach to any 
Birkhoff subcategory of a ``sufficiently good'' algebraic category. Here ``sufficiently good'' could mean being a semi-abelian 
variety of universal algebras \cite{BournJanelidze}, for instance, yielding many examples of interest in algebra.

\section{The case of nullifications}
\label{sec:nullification}
The results of the previous section apply to nullification functors. Let $\mathcal C$ be a semi-abelian category,
$A$ an object in $\mathcal C$, and define $\mathcal X \subset \mathcal C$ to be the (replete)
reflective subcategory of \emph{$A$-null} objects, i.e. of those objects $Z$ such that $\text{Hom}(A, Z) = 0$.

When it exists, the associated localization functor is written $P_A$ and called $A$-\emph{nullification} (or $A$-\emph{periodization}).
The construction is due to Bousfield in a homotopical setting, and can be found for example in Hirschhorn's \cite{MR1944041}, a reference in 
an algebraic context is Casacuberta, Peschke, and Pfenniger's \cite[Theorem~1.4]{MR1170580}. In all cases $P_A X$ is constructed
as a transfinite filtered colimit of iterated quotients of all morphisms from $A$. A cardinality argument is invoked to explain when one can stop the iteration.

In the recent preprint \cite{MSS} Monjon, Scherer and Sterck gave in Proposition $2.7$ an explicit construction of the nullification functor 
in the (semi-abelian) category of crossed modules. By looking at the arguments in their proof one realizes that these still apply to any semi-abelian 
variety of universal algebras \cite{BournJanelidze}. These are precisely those varieties (= finitary equational classes) whose algebraic theories have 
a unique constant $0$, $n\ge 1$ binary terms $\alpha_i(x,y)$ and one $(n+1)$-ary term $\beta$ satisfying the identities $\alpha_i(x,x)=0$ 
(for $i\in \{1, \cdots, n\}$), $\beta (\alpha_1(x,y), \cdots , \alpha_n(x,y), y)= x$. For example, in the case of the variety of groups, by using the 
multiplicative notation for the group operation, one can choose $0=1$, $\alpha_1(x,y)= x\cdot  y^{-1}$ and $\beta (x,y) = x\cdot y$. Note that, 
for a variety of universal algebras, being homological or being semi-abelian are equivalent properties, since a variety is always Barr-exact and cocomplete.
We work here with \emph{sets} equipped with finitary operations satisfying a \emph{set} of identities, so set-theoretic arguments are available.
Moreover any variety of universal algebras is cocomplete. Hence the proof of \cite[Proposition~2.7]{MSS} applies.

\begin{proposition}
\label{prop:nullification}
Let $\mathcal C$ be a semi-abelian variety of universal algebras and $A$ an object of $\mathcal C$. Then
the $A$-nullification functor $P_A$ exists and the coaugmentation morphism $\eta_X\colon X \rightarrow P_A X$ is a normal 
epimorphism, for any object $X$.
\end{proposition}

\begin{proof}
We only need to note that the construction yields a surjective coaugmentation morphism, which is thus a regular epimorphism.
The semi-abelian assumption on $\mathcal C$ then implies that $\eta_X\colon X \rightarrow P_A (X)$ actually  is a normal epimorphism.
\end{proof}

We show now that in the presence of fiberwise localization, nullification functors are conditonally flat, in fact even semi-left-exact. Let us
write $\overline P_A (X)$ for the kernel of the $A$-nullification $\eta_X \colon X \rightarrow P_A(X)$.
The equivalent characterization from Proposition~\ref{lem:barPAPA}(2) that $P_A (\overline P_A (X)) = 0$ for any object $X$ in~$\mathcal C$
is an algebraic analogue of Farjoun's \cite[Theorem~1.H.2]{Dror}.


\begin{corollary}
\label{thm:nullificationproper}
Consider a nullification functor $P_A$ on a semi-abelian variety of universal algebras $\mathcal C$, and assume that $P_A$ admits a
functorial fiberwise localization. Then $P_A$ is semi-left-exact. In particular $P_A$ is conditionally flat.
\end{corollary}

\begin{proof}
In view of Proposition~\ref{lem:barPAPA} it is sufficient to verify one of the equivalent conditions. By definition of $A$-local
objects it is easy to see that they are closed under extensions. Hence $P_A$, which exists by Proposition~\ref{prop:nullification}, is
semi-left-exact, a property which is stronger than admissibility for all regular epimorphisms. We conclude by Proposition~\ref{prop:condiflatfiberwise2}
that $P_A$ is conditionally flat.
\end{proof}

\section{A model categorical interpretation}
\label{sec:proper}
In this article we chose to study how pullbacks of exact sequences behave and in the previous sections we related this to semi-left-exactness,
a  stronger admissibility property (preservation of pullbacks along \emph{any} morphism between local objects versus preservation of pullbacks 
along any regular epimorphism between local objects).
From a model theoretic perspective, this corresponds to right properness as we explain next.

Any category with finite limits and colimits admits a \emph{discrete} model structure where weak equivalences are isomorphisms and 
all morphisms are fibrations and cofibrations. This easy observation has been already made by Bousfield, \cite[Examples~2.3]{MR57:17648},
who also constructed new model structures where the class of weak equivalences is  $\mathcal E(L)$, all morphisms inverted
by a localization functor $L$, i.e., $L$-equivalences. Cofibrations do not change and the class of fibrations 
coincides now with $\mathcal M(L)$ (using the notation from Section~\ref{sec:factorization}). This is not immediately obvious as
we require the lift to be unique in a factorization system, but not in a model category. The reason is that the model categorical lift
is unique up to homotopy in the associated homotopy factorization system and in the discrete setting ``unique up to homotopy"
means unique.

This process is called left Bousfield localization, we cite Salch's \cite[Proposition~3.5]{MR3606285} for a statement in the line of the present work. Our final propositions are just reformulations of the fact that a semi-left-exact reflection is also characterised by the property that for the induced factorization system $(\mathcal E, \mathcal M)$ the morphisms in $\mathcal E$ are stable under pullback along morphisms in $\mathcal M$.

\begin{proposition}
\label{prop:local} {\rm [Salch]}
Let $\mathcal C$ be a finitely cocomplete, finitely well-complete category and $L$ a localization functor. There is an $L$-local model structure on 
$\mathcal C$ where weak equivalences are the $L$-equivalences $\mathcal E(L)$, all morphisms are cofibrations, and 
the class of fibrations is $\mathcal M(L)$.
\end{proposition}

In a model category it is a direct consequence of the axioms that the pullback of a fibration is a fibration. But weak equivalences need not be preserved by pullbacks, not even by homotopy pullbacks. A model category
is \emph{right proper} if the pullback of any weak equivalence $X \rightarrow B$ along any fibration $E \twoheadrightarrow B$ 
is a weak equivalence. The discrete model structure is right proper since the pullback of an isomorphism along any map is an isomorphism.

Now, given a localization functor $L$ on $\mathcal C$, it is then a natural question to ask when the left Bousfield localized model structure
described in Proposition~\ref{prop:local} is again right proper. Rosick\'{y} and Tholen noticed in \cite[3.6]{MR2369170} that a result
by Cassidy, H\'ebert, and Kelly, \cite[Theorem~4.3]{MR779198}, allows one to characterize right proper localized model structures as
those corresponding to semi-left-exact reflections.

\begin{proposition}
\label{prop:semileftexact}{\rm [Rosick\'y-Tholen]}
Let $\mathcal C$ be a finitely complete category and $L$ a localization functor. The $L$-local model structure is right proper
if and only if $L$ is semi-left-exact.
\end{proposition}

Therefore Corollary~\ref{thm:nullificationproper} tells us that the $P_A$-local model structure is right proper,
an analogue of the well-known fact that nullification functors in spaces yield a right proper left Bousfield
localized model structure, \cite{MR1997044}, see also Wendt's \cite[Corollary~6.1]{MR3071143} for simplicial sheaves on a site.
However, in an algebraic setting, conditional flatness is different from right properness because pulling back an $L$-equivalence
along a regular epimorphism is not as general as pulling back along an arbitrary fibration, i.e. an arbitrary morphism in
the localized model structure.


\bibliographystyle{amsplain}

\def\cprime{$'$}
\providecommand{\bysame}{\leavevmode\hbox to3em{\hrulefill}\thinspace}
\providecommand{\MR}{\relax\ifhmode\unskip\space\fi MR }
\providecommand{\MRhref}[2]{%
  \href{http://www.ams.org/mathscinet-getitem?mr=#1}{#2}
}
\providecommand{\href}[2]{#2}



\end{document}